\documentclass[11pt]{amsart}

\usepackage{a4wide}
\usepackage{amssymb}
\usepackage{latexsym}
\usepackage{amsmath}
\usepackage{amsthm}
\usepackage{amsfonts}
\usepackage{color}
\usepackage{pdfsync}
\usepackage{epsfig}
\usepackage{enumitem}

\theoremstyle{plain}
\newtheorem{defi}{Definition}[section]

\newtheorem{teo}[defi]{Theorem}

\newtheorem{lema}[defi]{Lemma}

\newtheorem{remark}[defi]{Remark}

\theoremstyle{definition}

\theoremstyle{remark}

\numberwithin{equation}{section}

\begin{document}

\title[]{
A regularity result for a class of non-uniformly elliptic operators}

\author[]{Fausto Ferrari}
\address{\noindent
Fausto Ferrari: Dipartimento di Matematica, Universit\`a di Bologna,
Piazza di Porta S.Donato 5, Bologna, Italy.
\newline \textit{Email address:} {\tt fausto.ferrari@unibo.it}
}

\author[]{Giulio Galise}
\address{Giulio Galise:
Dipartimento di Matematica Guido Castelnuovo, Sapienza Universit\`a di Roma,
Piazzale Aldo Moro 5, Roma, Italy.
\newline \textit{Email address:} {\tt galise@mat.uniroma1.it}
}

\thanks{F. F. was partially supported by INDAM-GNAMPA 2019 project: {\it Propriet\`a di regolarit\`a delle soluzioni viscose con applicazioni a problemi di frontiera libera}  and INDAM-GNAMPA 2020 project: {\it Metodi di viscosit\`a e applicazioni a problemi non lineari con debole ellitticit\`a
.}}
\thanks{G. G. was partially supported by INDAM-GNAMPA 2019 project: {\it Problemi differenziali per operatori fully nonlinear fortemente degeneri}  and INDAM-GNAMPA 2020 project: {\it Problemi asintotici per EDP nonlineari e Mean Field Games.}}


\date{\today}

\begin{abstract} 
We obtain an explicit H\"older regularity result for viscosity solutions of a class of second order fully nonlinear equations  leaded by operator that are neither convex/concave nor uniformly elliptic.
\end{abstract}

\maketitle

\section{Introduction}
This note deals with the H\"older continuity issue of solutions of degenerate elliptic equations
of the form 
\begin{equation}\label{equa}
\mathcal{M}_{\mathbf{a}}(D^2u)=f(x)\quad\;\text{in $\Omega$,}
\end{equation}
where $\Omega\subset \mathbb{R}^N$ is a domain, $f\in C(\Omega)$ and $\mathcal{M}_{\mathbf{a}}$ is the weighted partial trace operator defined, for any symmetric matrix $X$, by the formula
\begin{equation}\label{0102eq1}
\mathcal{M}_{\mathbf{a}}(X)=\sum_{i=1}^Na_i\lambda_i(X).
\end{equation}
In \eqref{0102eq1}, $\lambda_1(X)\leq\ldots\leq\lambda_N(X)$ are the ordered eigenvalues of $X\in\mathbb S^N$ and $\mathbf{a}=\left\{a_1,\ldots,a_N\right\}$ is such that  $a_i\geq0$ for any $i=1,\ldots,N$. \\ It is plain that $\mathcal{M}_{\mathbf{a}}$ reduces to the classical Laplace operator when $\mathbf{a}=\left\{1,\ldots,1\right\}$ and that it fall out the class of uniformly elliptic operators as soon as $a_i=0$ for some $i=1,\ldots,N$. Such operators include, as particular cases, significant examples of degenerate operators, for instance   
$$
\mathcal{P}^-_k(X)=\sum_{i=1}^k\lambda_i(X)\quad\;\text{and}\quad\; \mathcal{P}^+_k(X)=\sum_{i=1}^k\lambda_{N-k+i}(X),
$$
which arise in the study of various geometric and elliptic problems, see e.g. \cite{AS, BGI, CDLV, CLN, CLN2, HMW,HL,HL2,S,W}, as well as the operators $\lambda_k(X)$, for some $k\in\left\{1,\ldots,N\right\}$, whose interest has been developed in the framework of differential games theory, see \cite{BER,BR,BCQ,BL}. 

In \cite{FV} the authors studied qualitative properties  of solutions of \eqref{equa} under the further assumption that $a_1>0$ and $a_N>0$, having in mind as prototype the Isaacs operator
\begin{equation}\label{Isaacs}
\lambda_1(X)+\lambda_N(X)=\min_{|\xi|=1}\max_{|\eta|=1}\left(\left\langle X\xi,\xi\right\rangle+\left\langle X\eta,\eta\right\rangle\right),
\end{equation}
which is neither uniformly elliptic nor convex/concave (in dimension $N\geq3$). Among  other results, they in particular obtained an Alexandov-Bakelman-Pucci (ABP) type inequality following the scheme of the proof showed in \cite{CC},  starting from the fact  (see \cite[Section 2.2]{CC} for the notation) that
\begin{equation}\label{CaffCab}
\mathcal{M}_{\mathbf{a}}\in \underline{\mathcal{S}}(\frac{a_*}{N},|\mathbf{a}|_1,f)\cap \overline{\mathcal{S}}(\frac{a_*}{N},|\mathbf{a}|_1,f),
\end{equation}
where $a_*=\min\left\{a_1,a_N\right\}$ and $|\mathbf{a}|_1=a_1+\cdots+a_N$. As a byproduct they obtain, in the same way of the uniformly elliptic case, that viscosity solutions of \eqref{equa} are $C_{\text{loc}}^{0,\alpha}(\Omega)$, where $\alpha\in(0,1)$, which is not explicitly known,  depends on the constant that appears in the ABP estimate. They did not obtain any further result about a possible lower bound on $\alpha$ or, possibly, a sharper result about the regularity of solutions due to the lack of structure in the nonlinear equation.

The goal of this note is to provide an explicit lower bound for $\alpha$, only depending  on $a_1$ and $a_N$. Applying the Ishii-Lions  approach to the problem (see \cite{IL}), we manage to prove that viscosity solutions of \eqref{equa} are $C^{0,\beta}_{\text{loc}}(\Omega)$, where 
\begin{equation}\label{beta}
\beta=1-\frac{a_1+a_N}{\left(\sqrt{a_1}+\sqrt{a_N}\right)^2}\,.
\end{equation}
From this we infer that $\alpha\geq\beta$ and, concerning the main example \eqref{Isaacs}, we in particular obtain that $\alpha\geq \frac{1}{2}$.\\
It is worth to point out that the fundamental assumption in the strategy of Ishii-Lions, in order to prove the Lipschitz continuity of solutions, is the uniformly ellipticity of the equation which clearly is outside our setting. Nevertheless, using the assumption $a_1>0$ and $a_N>0$, we are still able to detect some useful information encoded in the structure of the operator, so leading to the $\beta$-regularity of solutions, where $\beta$, defined in \eqref{beta}, is strictly less than one.\\
In addition, this approach can be applied to a larger class of operators depending to the first order term as well. \\
Thus, for stating our main result, we introduce the class of the equations we are going to consider.
Let
\begin{equation}\label{mainequation}
\mathcal{M}_{\mathbf{a}}(D^2u)+H(\nabla u)=f(x)\quad\;\text{in $\Omega$},
\end{equation}
where $\Omega\subseteq\mathbb R^N$ is a domain, $f$ is continuous in $\Omega$ and $
\mathcal{M}_{\mathbf{a}}$ is the fully nonlinear operator that we have introduced in \eqref{0102eq1}.\\
 Our  assumptions are:
\begin{enumerate}[label=(H\arabic*)]
\item\label{H1} $\mathcal{M}_{\mathbf{a}}\in\mathcal A=\{\mathcal{M}_{\mathbf{a}}(X):=\sum_{i=1}^Na_i\lambda_i(X):\:\: a_i\geq0,\: i=1,\dots,N,\:\: a_1>0,\:a_N>0,\:X\in\mathbb S^N \}.$
\item\label{H2} $H\in C(\mathbb R^N)$ and there exists a nonnegative constant $C_H$ such that 
\begin{equation}\label{assumptionH2}
|H(p+q)-H(p)|\leq C_H(1+|p|+|q|)|q|
\end{equation}
 for every $p,q\in\mathbb R^N$.
\end{enumerate}
 A typical example of $H$ satisfying \eqref{assumptionH2} is  $H(p)=A|p|^2+B|p|^\tau$, where $\tau\in[0,2]$ and $A,B\in\mathbb R$.  Although we shall allow  $H$ to have a  quadratic growth in the gradient variable, the prototype equation to be kept in mind it is still the one obtained when $H\equiv0$, e.g.
$$
a_1\lambda_1(D^2u)+a_N\lambda_N(D^2u)=f(x)\quad\;\text{in $\Omega$},
$$
 with $a_1,a_N>0$.

Now, we are in position to state our main result.
 \begin{teo}\label{teo}
Let $u\in C(\Omega)$ be a viscosity solution of \eqref{mainequation}. If \ref{H1}-\ref{H2} hold, then 
$$u\in C_{\text{loc}}^{0,\beta}(\Omega)\quad,\qquad\beta=1-\frac{a_1+a_N}{\left(\sqrt{a_1}+\sqrt{a_N}\right)^2}
  $$    
and the following estimate holds: for any $\omega\subset\subset\omega'\subset\subset\Omega$ one has  
\begin{equation*}
\left\|u\right\|_{C^{0,\beta}(\omega)}\leq C=C\left(a_1,a_N,\text{dist}(\omega,\omega'), C_H, \left\|u\right\|_{L^\infty(\omega')},\left\|f\right\|_{L^\infty(\omega')}\right).
\end{equation*}
\end{teo}
The main consequence of Theorem \ref{teo} is a lower bound of the expected regularity of viscosity solutions to a large class of operators that are not uniformly elliptic.  We point out that very few results are known about the sharp regularity of solution to fully nonlinear equations that are not convex/concave and that are not uniformly elliptic. In particular we recall the fundamental result \cite{NV}. Concerning the regularity issues of viscosity solutions of degenerate equations closely related to ours, we refer to \cite{CDLP, Fe, FeVe,FV}.

We conclude this introduction by pointing out that if we drop the assumption \ref{H1}, in the sense that $a_1=a_N=0$, then there exist viscosity solutions of 
$$
a_2\lambda_2(D^2u)+\ldots+a_{N-1}\lambda_{N-1}(D^2u)=0\quad\text{in $B_1$},
$$
 which doesn't belong to any $C_{\text{loc}}^{0,\alpha}(B_1)$ for  $\alpha\in (0,1)$, even if $a_i>0$ for any $i=2,\ldots,N-1$. We present a simple example at the end of this note.

\section{H\"older regularity}

We start with the following elementary lemma.

\begin{lema}\label{lema}
Let $\delta>0$ and let $A,B,C,D$ be nonnegative constants such that $A\in(0,1)$ and $C>0$. Then there exists $\varphi\in C^2((0,\delta])\cap C\left([0,\delta]\right)$, depending on $A,B,C,D,\delta$, which is a positive solution  of 
\begin{equation}\label{0102eq2}
\varphi''(r)+\left(\frac{A}{r}+B\right)\varphi'(r)=-C\qquad r\in(0,\delta]
\end{equation}
and satisfies the following conditions:
\begin{align}
&\varphi'(r)>0 \quad\text{and}\quad\varphi''(r)<0\quad\text{for any}\quad r\in(0,\delta]\label{1}\\
&\varphi''(r)-\frac{\varphi'(r)}{r}\leq-C\quad\text{for any}\quad r\in(0,\delta]\label{2}\\
&\varphi(\delta)\geq D\label{3}\\
&\sup_{0<r\leq\delta}\frac{\varphi(r)}{r^{1-A}}<+\infty\label{4}.
\end{align}
\end{lema}
\begin{proof}
By a straightforward computation, for any $K\in\mathbb R$ the function
$$
\varphi(r)=\int\limits_0^r\psi(s)\,ds\qquad\text{where}\qquad\psi(s)=\frac{e^{-Bs}}{s^A}\left(K-C\int\limits_0^st^Ae^{Bt}\,dt\right)
$$
is solution of \eqref{0102eq2}. Pick $K=K(A,B,C,D,\delta)$ such that $\psi(\delta)=\frac D\delta$. Hence $\varphi'(r)>0$ for $r\in(0,\delta]$ and, since $\varphi(0)=0$, then $\varphi(r)>0$ for any $r\in(0,\delta]$. Moreover,  just using the equation \eqref{0102eq2}, we infer that  \eqref{1} holds. Condition \eqref{2} easily follows by \eqref{1} and again using \eqref{0102eq2}.  Since $\psi$ is a decreasing function in $(0,\delta]$, then
$$
\varphi(\delta)=\int\limits_0^\delta\psi(s)\,ds\geq\delta\psi(\delta)=D
$$
by the choice of $K$. This shows \eqref{3}. To conclude it is sufficient to observe that for any $r\in(0,\delta]$ 
$$
\varphi(r)\leq K\int\limits_0^r\frac{1}{s^A}\,ds=\frac{K}{1-A}r^{1-A}.
$$
\end{proof}

\begin{proof}[Proof of Theorem \ref{teo}]
Take $\delta>0$ small enough such that $\omega_{2\delta}=\left\{x\in\mathbb R^N:\,\text{dist}(x,\omega)<2\delta\right\}\subset\omega'$. Fix $z\in\omega$ and let
$$
\Delta_z=\left\{(x,y)\in\Omega\times\Omega:\;|x-y|<\delta,\,|x-z|<\delta\right\}.
$$
Note that if $(x,y)\in\Delta_z$ then both $x$ and $y$ belongs in particular to $\omega'$. For $(x,y)\in\overline\Delta_z$ let
$$
\phi(x,y)=u(x)-u(y)-\varphi(|x-y|)-L|x-z|^2,
$$
where $\varphi(r)$ is, for $r\in[0,\delta]$, the function provided by Lemma \ref{lema} and depending on the parameters $A,B,C,D$.\\ Let $A=1-\beta\in(0,1)$. We claim that for an appropriate choice of $B,C,D$ and $L$, then 
\begin{equation}\label{2eq1}
\max_{(x,y)\in\overline\Delta_z}\phi(x,y)\leq0.
\end{equation}
This will implies the desired result, taking first $x=z$, then making $z$ vary and using \eqref{4}.

Set
\begin{equation}\label{2eq3}
\begin{split}
L&=\frac{2\left\|u\right\|_{L^\infty(\omega')}}{\delta^2}\\
D&=2\left\|u\right\|_{L^\infty(\omega')}\\
C&=\frac{2\left(L\left(|\mathbf{a}|_1+C_H\delta(1+2L\delta)\right)+\left\|f\right\|_{L^\infty(\omega')}+1\right)}{\left(\sqrt{a_1}+\sqrt{a_N}\right)^2}\\
B&=\frac{2L\delta C_H}{\left(\sqrt{a_1}+\sqrt{a_N}\right)^2}\,.
\end{split}
\end{equation}
By contradiction we suppose that \eqref{2eq1} does not hold. Let $(\hat x,\hat y)\in\overline\Delta_z$
such that 
\begin{equation}\label{2eq2}
\max_{(x,y)\in\overline\Delta_z}\phi(x,y)=\phi(\hat x,\hat y)>0
\end{equation}
By \eqref{2eq2} it is plain that $\hat x\neq\hat y$. Moreover, using \eqref{2eq3}, we exclude that $|\hat x-\hat y|=\delta$ or $|\hat x-z|=\delta$. Hence  $(\hat x,\hat y)\in\Delta_z$. By a standard result in theory of viscosity solutions, see \cite[Theorem 3.2 and Remark 3.8]{CIL}, for any $\varepsilon>0$ there exist matrices $X_\varepsilon,Y_\varepsilon\in\mathbb S^N$ such that
\begin{equation}\label{2eq4}
\left(\nabla\varphi(|\hat x-\hat y|)+2L(\hat x-z),X_\varepsilon+2LI\right)\in\overline{J}^{2,+} u(\hat x)
\end{equation}
\begin{equation}\label{2eq5}
\left(\nabla\varphi(|\hat x-\hat y|),Y_\varepsilon\right)\in\overline{J}^{2,-} u(\hat y)
\end{equation}
and 
\begin{equation}\label{2eq6}
\left(\begin{array}{cc} X_\varepsilon & 0\\
0 & -Y_\varepsilon\end{array}\right)\leq\left(\begin{array}{rr} \Theta_\varepsilon & -\Theta_\varepsilon\\ -\Theta_\varepsilon & \Theta_\varepsilon\end{array}\right),
\end{equation}
where $I\in \mathbb S^N$ is the identity matrix and $\Theta_\varepsilon\in \mathbb S^N$ is given by
$$
\Theta_\varepsilon=\varphi''(|\hat x-\hat y|)\left(1+2\varepsilon\varphi''(|\hat x-\hat y|)\right)P+\frac{\varphi'(|\hat x-\hat y|)}{|\hat x-\hat y|}\left(1+2\varepsilon\frac{\varphi'(|\hat x-\hat y|)}{|\hat x-\hat y|}\right)(I-P)
$$
where $P=\frac{\hat x-\hat y}{|\hat x-\hat y|}\otimes\frac{\hat x-\hat y}{|\hat x-\hat y|}$. Note that the eigenvalues of $\Theta_\varepsilon$ are
 $\varphi''(|\hat x-\hat y|)\left(1+2\varepsilon\varphi''(|\hat x-\hat y|)\right)$, which is simple, and $\frac{\varphi'(|\hat x-\hat y|)}{|\hat x-\hat y|}\left(1+2\varepsilon\frac{\varphi'(|\hat x-\hat y|)}{|\hat x-\hat y|}\right)$ with multiplicity $N-1$. In view of \eqref{2} we can assume that, for $\varepsilon$ sufficiently small, one has 
$$
\lambda_1(\Theta_\varepsilon)=\varphi''(|\hat x-\hat y|)\left(1+2\varepsilon\varphi''(|\hat x-\hat y|)\right)
$$
and that 
$$
\lambda_2(\Theta_\varepsilon)=\ldots=\lambda_{N}(\Theta_\varepsilon)=\frac{\varphi'(|\hat x-\hat y|)}{|\hat x-\hat y|}\left(1+2\varepsilon\frac{\varphi'(|\hat x-\hat y|)}{|\hat x-\hat y|}\right).
$$
Using \eqref{2eq4}-\eqref{2eq5} and the equation \eqref{mainequation} we then obtain
\begin{equation*}
\begin{split}
-2\left\|f\right\|_{L^\infty(\omega')}&\leq a_1\lambda_1(X_\varepsilon)+a_N\lambda_N(X_\varepsilon)-a_1\lambda_1(Y_\varepsilon)-a_N\lambda_N(Y_\varepsilon)\\
&\qquad+\sum_{i=2}^{N-1}a_i\left(\lambda_i(X_\varepsilon)-\lambda_i(Y_\varepsilon)\right)+2L|\mathbf{a}|_1\\
&\qquad+H\left(\nabla\varphi(|\hat x-\hat y|)+2L(\hat x-z)\right)-H\left(\nabla\varphi(|\hat x-\hat y|)\right)\,.
\end{split}
\end{equation*}
Since $X_\varepsilon\leq Y_\varepsilon$ and $a_i\geq0$, then using \eqref{assumptionH2} and \eqref{1} we have
\begin{equation}\label{2eq8}
\begin{split}
-2\left\|f\right\|_{L^\infty(\omega')}&\leq a_1\lambda_1(X_\varepsilon)+a_N\lambda_N(X_\varepsilon)-a_1\lambda_1(Y_\varepsilon)-a_N\lambda_N(Y_\varepsilon)\\
&\qquad+2L|\mathbf{a}|_1+2L\delta C_H(1+\varphi'(|\hat x-\hat y|)+2L\delta)\,.
\end{split}
\end{equation}
In order to reach a contradiction we now estimate the right hand side on \eqref{2eq8} using the inequality
\begin{equation}\label{2eq9}
\left(\begin{array}{cc} X_\varepsilon & 0\\
0 & -Y_\varepsilon\end{array}\right)\left(\begin{array}{c}v\\ w\end{array}\right)\cdot\left(\begin{array}{c}v\\ w\end{array}\right)\leq\left(\begin{array}{rr} \Theta_\varepsilon & -\Theta_\varepsilon\\ -\Theta_\varepsilon & \Theta_\varepsilon\end{array}\right)\left(\begin{array}{c}v\\ w\end{array}\right)\cdot\left(\begin{array}{c}v\\ w\end{array}\right)\quad\forall v,w\in\mathbb R^N
\end{equation}
and choosing in a suitable way $v,w$.\\
With the choice $$v=\sqrt{a_1}\,\frac{\hat x-\hat y}{|\hat x-\hat y|}\;,\quad w=-\sqrt{a_N}\,\frac{\hat x-\hat y}{|\hat x-\hat y|},$$ then \eqref{2eq9} yields
\begin{equation}\label{2eq10}
\begin{split}
a_1\lambda_1(X_\varepsilon)-a_N\lambda_N(Y_\varepsilon)&\leq a_1X_\varepsilon\frac{\hat x-\hat y}{|\hat x-\hat y|}\cdot\frac{\hat x-\hat y}{|\hat x-\hat y|}-a_NY_\varepsilon\frac{\hat x-\hat y}{|\hat x-\hat y|}\cdot\frac{\hat x-\hat y}{|\hat x-\hat y|}\\
&\leq\left(\sqrt{a_1}+\sqrt{a_N}\right)^2\Theta_\varepsilon\frac{\hat x-\hat y}{|\hat x-\hat y|}\cdot\frac{\hat x-\hat y}{|\hat x-\hat y|}\\
&=\left(\sqrt{a_1}+\sqrt{a_N}\right)^2\varphi''(|\hat x-\hat y|)\left(1+2\varepsilon\varphi''(|\hat x-\hat y|)\right).
\end{split}
\end{equation}
On the other hand, taking 
$$v=\sqrt{a_N}\,\xi\;,\quad w=0,$$
where $|\xi|=1$ and  $X_\varepsilon\xi=\lambda_N(X_\varepsilon)\xi$, we have
\begin{equation}\label{2eq11}
a_N\lambda_N(X_\varepsilon)\leq a_N\,\Theta_\varepsilon\xi\cdot\xi\leq a_N\frac{\varphi'(|\hat x-\hat y|)}{|\hat x-\hat y|}\left(1+2\varepsilon\frac{\varphi'(|\hat x-\hat y|)}{|\hat x-\hat y|}\right).
\end{equation}
In a similar way we also obtain that 
\begin{equation}\label{2eq12}
-a_1\lambda_1(Y_\varepsilon)\leq a_1\Theta_\varepsilon\xi\cdot\xi\leq a_1\frac{\varphi'(|\hat x-\hat y|)}{|\hat x-\hat y|}\left(1+2\varepsilon\frac{\varphi'(|\hat x-\hat y|)}{|\hat x-\hat y|}\right).
\end{equation}
Putting together \eqref{2eq8}, \eqref{2eq10}-\eqref{2eq12} we infer that 
\begin{equation}\label{2eq13}
\begin{split}
-2\left\|f\right\|_{L^\infty(\omega')}&\leq\left(\sqrt{a_1}+\sqrt{a_N}\right)^2\varphi''(|\hat x-\hat y|)+\left(\frac{a_1+a_N}{|\hat x-\hat y|}+2 L\delta C_H\right)\varphi'(|\hat x-\hat y|)\\&\qquad+2L\left(|\mathbf{a}|_1+C_H\delta(1+2L\delta)\right)\\&\qquad+2\varepsilon\left[\left(\sqrt{a_1}+\sqrt{a_N}\right)^2(\varphi''(|\hat x-\hat y|))^2+(a_1+a_N)\left(\frac{\varphi'(|\hat x-\hat y|)}{|\hat x-\hat y|}\right)^2\right]\,.
\end{split}
\end{equation}
By  \eqref{2eq3} and Lemma \ref{lema},  the function $\varphi(r)$ is solution, for $r\in(0,\delta]$, of the ordinary differential equation
\begin{equation}\label{2eq14}
\begin{split}
\left(\sqrt{a_1}+\sqrt{a_N}\right)^2\varphi''(r)&+\left(\frac{a_1+a_N}{r}+2 L\delta C_H\right)\varphi'(r)=\\&=-2\left(L\left(|\mathbf{a}|_1+C_H\delta(1+2L\delta)\right)+\left\|f\right\|_{L^\infty(\omega')}+1\right).
\end{split}
\end{equation}
Coupling \eqref{2eq13}-\eqref{2eq14} then
 $$
1\leq\varepsilon\left[\left(\sqrt{a_1}+\sqrt{a_N}\right)^2(\varphi''(|\hat x-\hat y|))^2+(a_1+a_N)\left(\frac{\varphi'(|\hat x-\hat y|)}{|\hat x-\hat y|}\right)^2\right]
$$  
 leading  a contradiction for $\varepsilon$ small enough.
\end{proof}

\begin{remark}
We note that in the case $H\equiv0$, the function $\varphi(r)$ used in the proof of Theorem \ref{teo} is more explicit, in fact it is given by $\varphi(r)=Ar^\beta-Br^2$ for a suitable choice of $A,B>0$.
\end{remark}

\subsection{Lack of regularity} Let $N\geq3$ and consider the equation
\begin{equation}\label{eq1}
a_2\lambda_2(D^2u)+\ldots+a_{N-1}\lambda_{N-1}(D^2u)=0\quad\text{in $B_1$}.
\end{equation}
We are going to exhibit a continuous function $u$ which is  solution of \eqref{eq1} for any $a_i\geq0$ and $i=2,\ldots,N-1$, but which does not belong to $C^{0,\alpha}_{\text{loc}}(B_1)$ for any possible choice of $\alpha\in(0,1]$.

Let $f:(-1,1)\mapsto\mathbb R$ be the function
$$
f(t)=\left\{
\begin{array}{cl}
\frac{1}{2-\log|t|} & \text{if $t\neq0$}\\
0 & \text{if $t=0$}
\end{array}
\right.
$$
and consider it as a function of $N$ variables just by setting $u(x)=f(x_1)$ for $x\in B_1$. It is clear that $u\in C(B_1)$ but $u\notin C^{0.\alpha}_{\text{loc}}(B_1)$ for any $\alpha\in(0,1]$. 
We claim that $u$ is a viscosity solution of \eqref{eq1}.

The function $u$ is smooth for $x\in B_1\backslash\left\{x\in B_1:\,x_1=0\right\}$ and $$D^2u(x)=\text{diag}(f''(x_1),0,\ldots,0).$$ Since $f''(t)\leq0$ for any $t\in(-1,1)\backslash\left\{0\right\}$, we infer that $u$ is in fact a classical solution of \eqref{eq1} in the set $B_1\backslash\left\{x\in B_1:\,x_1=0\right\}$.\\
Now we prove that $u$ satisfies (in viscosity sense) the equation \eqref{eq1} also in $\left\{x\in B_1:\,x_1=0\right\}$. 
For $x\in\mathbb R^N$ such that $x_1=0$ we adopt the notation $x=(0,x')$ with $x'\in\mathbb R^{N-1}$.\\ Let $x_0=(0,x'_0)\in B_1$. Since there are no test functions $\phi\in C^2(B_1)$ touching $u$ from above  at $x_0$,  we infer that $u$ is a viscosity subsolution of \eqref{eq1}.
As far as the supersolution property is concerned, it is sufficient to prove that if $\phi\in C^2(B_1)$ is  such that
\begin{equation}\label{eq2}
0=u(0,x'_0)=\phi(0,x'_0)\quad\text{and}\quad u(x)\geq\phi(x)\;\;\;\forall x\in B_1
\end{equation}
then $\lambda_{N-1}(D^2\phi(0,x'_0))\leq0$.\\
Set $\psi(x')=\phi(0,x')$ for $|x'|<1$. From \eqref{eq2} we deduce that $\psi(x')$ attains its maximum at $x'_0$. Hence
\begin{equation}\label{eq3}
\left\langle D^2\psi(x'_0)v,v\right\rangle\leq0\quad\;\forall v\in\mathbb R^{N-1}.
\end{equation}
Using the Courant-Fischer formula
$$
\lambda_{N-1}(D^2\phi(0,x'_0))=\min_{\dim W=N-1}\max_{\stackrel{w\in W}{|w|=1}}\left\langle D^2\phi(0,x'_0)w,w\right\rangle,
$$ 
with the particular choice of $W=\left\{\left(0,v\right):\,v\in\mathbb R^{N-1}\right\}$, and \eqref{eq3} we then obtain
$$
\lambda_{N-1}(D^2\phi(0,x'_0))\leq\max_{\stackrel{v\in\mathbb R^{N-1}}{|v|=1}}\left\langle D^2\psi(x'_0)v,v\right\rangle\leq0.
$$
This shows that $u(x)$ is a viscosity solution of \eqref{eq1}, for any  $a_i\geq0$ and $i=2,\ldots,N-1$.

\end{document}